\documentclass[preprint,12pt]{elsarticle}
\usepackage{amsmath,amssymb,amsthm,mathtools}
\usepackage{array}

\usepackage{graphicx}
\usepackage{hyperref}
\hypersetup{
	colorlinks=true,
	linkcolor=cyan,
	citecolor=cyan,
}
\usepackage{exscale}
\usepackage{algorithm,algpseudocode}

\newtheorem{theorem}{Theorem}[section]
\newtheorem{lemma}{Lemma}[section]
\newtheorem{con}{Conjecture}[section]
\newtheorem{cor}{Corollary}[section]
\newtheorem{rk}{Remark}[section]
\newtheorem{de}{Definition}[section]
\newtheorem{prop}{Proposition}[section]

\begin{document}
	\sloppy
	\begin{frontmatter}
		\title{Primitive Polynomials of the Form $g(x)+\lambda$ over Finite Fields: Non-Existence Results and Conjectures}
		  \author[1]{Avnish K. Sharma}
		  \ead{avnish.sharma@srcc.du.ac.in}
		\cortext[cor1]{Corresponding author}
		\address[1]{Department of Mathematics, Shri Ram College of Commerce, University of Delhi, New Delhi-110007, India}

\begin{abstract}
        In this paper, we investigate the existence of primitive polynomials over $\mathbb{F}_{q^n}$ whose constant term is a primitive element of $\mathbb{F}_{q^n}$. We prove that such polynomials do not exist if $q$ is odd, $q^n\equiv3\pmod{4}$, and the degree $m$ of the polynomial is odd. In particular, the polynomials $f(x)=g(x)+\lambda$, where $g(x)\in\mathbb{F}_q[x]$ satisfies $g(0)=0$ and $\lambda\in\mathbb{F}_{q^n}$ is primitive, cannot be primitive under the same conditions. Further, for the cubic polynomial $x^3+x^2+x+\lambda$, we establish non-existence results in characteristics $2$ and $3$. These results, in particular, provide counterexamples to previously proposed existence conjectures.

        We also study the family $x^p+x+\lambda$ over $\mathbb{F}_{p^n}$. For an odd prime $p$, we prove that, provided $\sum_{i=0}^{n-1}(-1)^i\lambda^{p^i}\neq0$, the polynomial $x^p+x+\lambda$ is irreducible over $\mathbb{F}_{p^n}$ if and only if $n$ is even. Motivated by this result and supported by computational evidence, we formulate a conjecture concerning the existence of such primitive polynomials.
\end{abstract}
		

\end{frontmatter}

\section{Introduction}\label{S1}

       Let $q$ be a prime power and $n$ be a positive integer. The symbol $\mathbb{F}_{q}$ denotes a finite field with $q$ elements, and the symbol $\mathbb{F}_{q^n}$ denotes the extension of $\mathbb{F}_{q}$ of degree $n$. An element $\alpha\in\mathbb{F}_{q^n}$ is called a primitive element if its multiplicative order is $q^n-1$. A monic polynomial $f(x)\in\mathbb{F}_{q^n}[x]$ of degree $m$ is called primitive if it is the minimal polynomial of a primitive element of $\mathbb{F}_{q^{mn}}$ over $\mathbb{F}_{q^n}$. Readers may refer to \cite{Nieder} for more details on this topic.
       
       Primitive polynomials play a fundamental role in finite field theory and have various applications in cryptography and coding theory. In particular, they are extensively used in pseudo-random number generation, including Linear Feedback Shift Register (LFSR)-based stream ciphers, as they produce sequences of maximal period.

       In \cite{TSR}, Ambrish Awasthi and Rajendra K. Sharma investigated a generalization of Linear Feedback Shift Registers, known as Transformation Shift Registers (TSRs). They showed that the existence of primitive TSRs is equivalent to the existence of a certain class of primitive polynomials over finite fields. Consequently, the problem of constructing primitive TSRs reduces to that of establishing the existence of these special primitive polynomials. They proposed the following conjectures.
  
    \begin{con}\upshape{\cite[Conjecture 4.1]{TSR}}\label{Con1.1}
        There exists a primitive polynomial $f(x)$ of degree $m$ over $\mathbb{F}_{q^n}$ of the following form $f(x)=g(x)+\lambda, \text{ for every prime }q, \text{ and } m,n\geq 2,$ where $g(x)\in\mathbb{F}_{q}[x]$ such that $g(0)=0$ and $\lambda$ is a primitive element in $\mathbb{F}_{q^n}$.
    \end{con}

    In particular, by taking $g(x)=x^3+x^2+x$ and $n=2$, they further proposed another conjecture.
    
    \begin{con}\upshape{\cite[Conjecture 9.1]{TSR}}\label{Con1.2}
        There exist primitive polynomials which have the form $x^3 + x^2 + x + \lambda$ over $\mathbb{F}_{q^2}$ for every prime $q$, where $\lambda$ is a primitive element in $\mathbb{F}_{q^2}$.
    \end{con}

    In this paper, we extend the non-existence analysis from odd prime $q$ to odd prime power $q$. We establish a non-existence result for primitive polynomials $f(x)$ of odd degree $m$ with primitive constant term over $\mathbb{F}_{q^n}$ for $q^n\equiv 3\pmod 4$. In particular, we prove that there are no primitive polynomials of the form $f(x)=g(x)+\lambda$ of odd degree $m$ over $\mathbb{F}_{q^n}$, where $g(x)\in\mathbb{F}_{q}[x]$ with $g(0)=0$ and $\lambda$ is a primitive element of $\mathbb{F}_{q^n}$, if $q^n\equiv 3\pmod 4$. This result, in particular, disproves Conjecture \ref{Con1.1} for every odd prime $q\equiv 3\pmod{4}$, by taking any odd $n\geq 3$ and any odd $m\geq 3$. Further, we establish another non-existence result for cubic polynomials of the form $x^3+x^2+x+\lambda$ over fields of characteristic $3$ and $2$. In particular, we disprove Conjecture \ref{Con1.2} for $q=3$ and $n=2$, and for $q=2$ and for any odd $n\geq 3$. Motivated by these non-existence results and computational evidence, we propose conjectures regarding the existence of primitive polynomials $f(x)=g(x)+\lambda$ and $x^3+x^2+x+\lambda$ over $\mathbb{F}_{q^n}$. Finally, we extend our study to the polynomials of the form $x^p+x+\lambda$ over $\mathbb{F}_{p^n}$, where $p$ is a prime. Using a result established by S.D. Cohen \cite[Theorem 1]{cohen1990primitive}, we prove that, if $p=2$, such primitive polynomials always exists for every $n\geq 1$. For odd primes $p$, we prove that if $\lambda\in\mathbb{F}_{p^n}$ satisfies the condition $\sum_{i=0}^{n-1}(-1)^i\lambda^{p^i}\neq 0$ and $n$ is even, then the polynomial $x^p+x+\lambda$ is always irreducible. In fact, we conjecture that, if $n$ is even, then there exists a primitive element $\lambda\in\mathbb{F}_{p^n}$ satisfying this condition such that the polynomial $x^p+x+\lambda$ is primitive over $\mathbb{F}_{p^n}$. In particular, we take $n=2$ and propose a stronger conjecture that for every odd prime $p$ and for every primitive $\lambda\in\mathbb{F}_{p^2}$, the polynomial $x^p+x+\lambda$ is primitive over $\mathbb{F}_{p^2}$. 

\section{A general non-existence result and existence conjectures for primitive polynomials $f(x)=g(x)+\lambda$}\label{S2}    

    In this section, we prove that there does not exist any primitive polynomial of the form $f(x)=g(x)+\lambda$, where $g(x)\in\mathbb{F}_q[x]$, $g(0)=0$ and $\lambda$ is a primitive element of $\mathbb{F}_{q^n}$, where $q$ is a prime power. We also propose two existence conjectures for such primitive polynomials. Before proving this result, we recall the following lemma.

    \begin{lemma}{\upshape\cite[Theorem 3.18]{Nieder}}\label{L2.1}
        Let $q$ be any prime power and let $f(x)\in\mathbb{F}_{q}[x]$ be a primitive polynomial of degree $m\geq 1$ over $\mathbb{F}_{q}$. Then $(-1)^mf(0)$ is a primitive element of $\mathbb{F}_{q}$.
    \end{lemma}
    
    We now state our main result of this section.

    \begin{theorem}\label{T2.1}
        Let $q$ be an odd prime power and let $m,n$ be positive integers. If $q^n\equiv3\pmod{4}$ and $m$ is odd, then there does not exist a primitive polynomial $f(x)$ of degree $m$ over $\mathbb{F}_{q^n}$ with primitive constant term.
    \end{theorem}

    \begin{proof}
        Suppose, to the contrary, that there exists a primitive polynomial $f(x)$ of degree $m$ over $\mathbb{F}_{q^n}$ such that the constant term $f(0)=\lambda$ is a primitive element of $\mathbb{F}_{q^n}$. By Lemma \ref{L2.1}, $(-1)^m\lambda=-\lambda$ is a primitive element of $\mathbb{F}_{q^n}$. Since $q$ is odd and $\lambda$ is primitive in $\mathbb{F}_{q^n}$, we have $-1=\lambda^{(q^n-1)/2}$, and consequently $-\lambda=\lambda^{(q^n+1)/2}$. Now, since $q^n\equiv 3\pmod{4}$, $(q^n+1)/2$ is even. Writing $(q^n+1)/2=2t$, we obtain $-\lambda=\lambda^{2t}=(\lambda^t)^2$. Thus, $-\lambda$ is a square in $\mathbb{F}_{q^n}$. However, no primitive element of $\mathbb{F}_{q^n}$ can be a square, and hence we arrive at a contradiction. Therefore, no such primitive polynomial $f(x)$ exists over $\mathbb{F}_{q^n}$.
    \end{proof}

        It is important to note that {\upshape Theorem \ref{T2.1}} does not require $f(x)$ to be of any special form. It only requires that the constant term of $f(x)$ be primitive and that the degree $m$ be odd. Hence, if the polynomial is of the form $f(x)=g(x)+\lambda$, where $g(x)\in\mathbb{F}_q[x]$ and $g(0)=0$, the proof remains valid. Therefore, in particular, we have the following result.

    \begin{cor}\label{Cor2.1}
        Let $q$ be an odd prime power and let $m,n\geq1$. If $q^n\equiv3\pmod{4}$ and $m$ is odd, then there does not exist a primitive polynomial $f(x)=g(x)+\lambda$ of degree $m$ over $\mathbb{F}_{q^n}$, where $g(x)\in\mathbb{F}_q[x]$, $g(0)=0$ and $\lambda$ is a primitive element of $\mathbb{F}_{q^n}$.
    \end{cor}

        We note that in the above result, we assume $q$ to be an odd prime power. However, Conjecture \ref{Con1.1} assumes only that $q$ is a prime. In particular, {\upshape Conjecture \ref{Con1.1}} of Ambrish Awasthi and Rajendra K. Sharma is false for any odd prime $q$ such that $q^n\equiv 3 \pmod 4$ and for every odd integer $m\geq 3$.

    \begin{cor}\label{Cor2.2}
        Let $q$ be an odd prime and let $n$ be a positive integer. If $q^n\equiv3\pmod{4}$ and $m\geq 3$ is odd, then there does not exist a primitive polynomial $f(x)=g(x)+\lambda$ of degree $m$ over $\mathbb{F}_{q^n}$, where $g(0)=0$ and $\lambda$ is a primitive element of $\mathbb{F}_{q^n}$.
    \end{cor}

    \begin{rk}\label{R2.1}
         We provide computational verification of {\upshape Corollary \ref{Cor2.2}} through {\upshape Algorithm \ref{Algo1}} for certain values of $(q,m,n)$. We have found that for each triplet $(q,m,n)$ listed in {\upshape Table \ref{table1}}, there is no primitive polynomial $f(x)=g(x)+\lambda$ with the desired properties. These computations were performed using {\upshape SageMath (Version 9.0) \cite{sagemath}}.

    \end{rk}

    \begin{algorithm}[h!]
    \caption{Exhaustive Search for Primitive Polynomials $f(x)=g(x)+\lambda$ over $\mathbb{F}_{q^n}$ with Primitive $\lambda\in\mathbb{F}_{q^n}$}
    \label{Algo1}
    \begin{algorithmic}[1]
        \Require $q,m,n$
        \Ensure The number of primitive polynomials
        $f(x)=g(x)+\lambda$ of degree $m$
        over $\mathbb{F}_{q^n}$ with $\lambda$ primitive in $\mathbb{F}_{q^n}$ and $g(x)\in\mathbb{F}_q[x]$ with $g(0)=0$

        \State Construct $F=\mathbb{F}_{q}$, $K=\mathbb{F}_{q^n}$, and the polynomial rings $F[x]$ and $K[x]$
        \State Set $\texttt{primitive}=0$

        \For{$\lambda\in K$}
            \If{$\mathrm{ord}(\lambda)=q^n-1$}
                \For{each monic $g(x)\in F[x]$ of degree $m$ with $g(0)=0$}
                    \State Define
                    $f(x)=g(x)+\lambda$
                    \If{$f(x)$ is primitive over $K$}
                        \State $\texttt{primitive}\gets\texttt{primitive}+1$
                    \EndIf
                \EndFor
            \EndIf
        \EndFor
        \State \textbf{return} $\texttt{primitive}$
    \end{algorithmic}
\end{algorithm}

    \begin{table}[h!]
        \centering
        \caption{Computational verification of Corollary \ref{Cor2.2} for several choices of $(q,m,n)$}
        \label{table1}
        \begin{tabular}{c|>{\centering\arraybackslash}p{4cm}|>{\centering\arraybackslash}p{5cm}}
        \hline
        $(q,m,n)$& \# Polynomials Tested& \# Primitive Polynomials Found\\
        \hline
        $(3,3,3)$ & $9 \times 12$ & 0\\
        $(3,5,3)$ & $81 \times 12$ & 0\\
        $(7,3,3)$ & $49 \times 108$ & 0\\
        $(7,5,3)$ & $2401 \times 108$ & 0\\
        $(19,3,3)$ & $361 \times 2268$ & 0\\
        $(23,3,3)$ & $529 \times 4680$ & 0\\
        $(31,3,3)$ & $961 \times 7920$ & 0\\
        \hline
        \end{tabular}
    \end{table}

    \begin{rk}\label{R2.2}
        The computations suggest that if either $m$ is even or $q^n\equiv1 \pmod 4$, then the desired primitive polynomials $f(x)=g(x)+\lambda$ may exist. For computational verification of this observation, we apply {\upshape Algorithm \ref{Algo1}} for several triplets $(q,m,n)$ for which either $m$ is even or $q^n\equiv1 \pmod 4$ and found that there exist primitive polynomials $f(x)=g(x)+\lambda$ with the desired properties. The results are tabulated in {\upshape Table \ref{table2}}.
    \end{rk}

    \begin{table}[h!]
        \centering
        \caption{Number of primitive polynomials $f(x)=g(x)+\lambda$ over $\mathbb{F}_{q^n}$ with desired properties for several triplets $(q,m,n)$ when either $m$ is even or $q^n\equiv1 \pmod 4$}
        \label{table2}
        \begin{tabular}{c|>{\centering\arraybackslash}p{4cm}|>{\centering\arraybackslash}p{5cm}}
        \hline
        $(q,m,n)$& \# Polynomials Tested& \# Primitive Polynomials Found\\
        \hline
        $(3,3,2)$ & $9 \times 4$ & 12\\
        $(3,4,3)$ & $27 \times 12$ & 54\\
        $(3,3,4)$ & $9 \times 32$ & 72\\
        $(5,2,2)$ & $5 \times 8$ & 16\\
        $(5,3,1)$ & $25 \times 2$ & 20\\
        $(5,3,2)$ & $25 \times 8$ & 72\\
        $(5,3,4)$ & $25 \times 192$ & 1408\\
        \hline
        \end{tabular}
    \end{table}

    \begin{rk}\label{R2.3}
        We also note that, in {\upshape Theorem \ref{T2.1}}, $q$ is assumed to be an odd prime power. However, if we assume $q=2^k,\, k\geq 1$, then the computations suggest that there exists a primitive polynomial of the form $f(x)=g(x)+\lambda$ of degree $m\geq 1$ with the desired properties. The computational verification for this observation is presented in {\upshape Table \ref{table3}} for several triplets $(q,m,n)$ using {\upshape Algorithm \ref{Algo1}}.
    \end{rk}

    Motivated by Theorem \ref{T2.1} and by the computational verifications in Remarks \ref{R2.2} and \ref{R2.3}, we propose the following conjectures.

    \begin{con}\label{Con2.1}
        Let $q$ be an odd prime power and let $m,n\geq 1$. If $m$ is even or $q^n\equiv1 \pmod{4}$, then there exists a primitive polynomial $f(x)$ of degree $m$ over $\mathbb{F}_{q^n}$ of the form $f(x)=g(x)+\lambda,$ where $g(x)\in\mathbb{F}_q[x]$ satisfies $g(0)=0$, and $\lambda$ is a primitive element of $\mathbb{F}_{q^n}$.
    \end{con}

    \begin{con}\label{Con2.2}
        Let $q=2^k,\, k\geq 1$, and let $m,n\geq 1$. Then there exists a primitive polynomial $f(x)$ of degree $m$ over $\mathbb{F}_{q^n}$ of the form $f(x)=g(x)+\lambda,$ where $g(x)\in\mathbb{F}_q[x]$ satisfies $g(0)=0$, and $\lambda$ is a primitive element of $\mathbb{F}_{q^n}$.
    \end{con}

    \begin{table}[h!]
        \centering
        \caption{The number of primitive polynomials
$f(x)=g(x)+\lambda$ over $\mathbb{F}_{q^n}$ for $q=2^k$}
        \label{table3}
        \begin{tabular}{c|>{\centering\arraybackslash}p{4cm}|>{\centering\arraybackslash}p{5cm}}
        \hline
        $(q,m,n)$& \# Polynomials Tested& \# Primitive Polynomials Found\\
        \hline
        $(2,2,2)$ & 4 & 2\\
        $(2,2,3)$ & 12 & 3\\
        $(2,3,2)$ & 8 & 2\\
        $(2,3,3)$ & 24 & 6\\
        $(2,4,2)$ & 16 & 4\\
        $(2,4,3)$ & 48 & 12\\
        $(2,5,2)$ & 32 & 10\\
        $(2,5,3)$ & 96 & 12\\
        $(4,5,3)$ & 9216 & 1682 \\
        $(8,3,2)$ & 2304 & 864\\
        $(16,2,3)$ & 27648 & 11808\\
        \hline
        \end{tabular}
    \end{table}   

\section{Non-existence results and existence conjectures for the cubic polynomial $x^3+x^2+x+\lambda$ }\label{S3}

    In this section, we establish the non-existence of a primitive cubic polynomial in the fields of characteristic $2$ and $3$. In particular, we prove that Conjecture \ref{Con1.2} is false if $q=3$ and $n=2$, or if $q=2$ and $n$ is odd. We divide these results into two subsections.

\subsection{When the fields are of characteristic $3$}

    In this subsection, we establish the following general non-existence result for the cubic polynomial $x^3+x^2+x+\lambda$ in the field of characteristic $3$.

    \begin{theorem}\label{T3.1}
        Let $q=3^k,\, k\geq 1$. Then for any $n\geq 1$, the polynomial $f(x)=x^3+x^2+x+\lambda$ is not primitive over $\mathbb{F}_{q^n}$, where $\lambda$ is a primitive element of  $\mathbb{F}_{q^n}$.  
    \end{theorem}

   To establish this result, we first recall the following definition of the discriminant of a polynomial and a proposition concerning its properties.

    \begin{de}\upshape{\cite[Definition 2.1.135]{mullenhandbook}}\label{D3.1}
    Let $f$ be a polynomial of degree $m$ in $\mathbb{F}_q[x]$ with leading coefficient $a$, and with roots $\alpha_1,\alpha_2,\ldots,\alpha_m$ in its splitting field, counted with multiplicity. The discriminant of $f$ is given by
    $$D(f)=a^{2m-2}\prod_{1\leq i<j\leq m}(\alpha_i-\alpha_j)^2.$$
\end{de}

\begin{prop}\label{P3.1}
    Let $\mathbb{F}_{q}$ be a finite field, and let $f(x)\in \mathbb{F}_{q}[x]$ be an irreducible cubic polynomial. Then the discriminant $D(f)$ is a square in $\mathbb{F}_{q}$.
\end{prop}
\begin{proof}
    Let $\beta$ be a root of a cubic irreducible polynomial $f(x)$. Since $f(x)$ is irreducible of degree $3$, its roots are $\beta,\beta^q,\beta^{q^2}$. Then the element 
    $$\delta=(\beta-\beta^{q})(\beta^{q}-\beta^{q^2})(\beta^{{q^2}}-\beta)$$
    is fixed under the Frobenius automorphism $\sigma:x\mapsto x^q$, that is, 
    $$\sigma(\delta)=(\beta^{q}-\beta^{{q^2}})(\beta^{{q^2}}-\beta)(\beta-\beta^{q})=\delta.$$ 
    Thus, $\delta\in \mathbb{F}_{q}$. On the other hand, by the definition of the discriminant, $$D(f)=(\beta-\beta^{q})^2(\beta-\beta^{{q^2}})^2(\beta^{{q}}-\beta^{q^2})^2=\delta^2.$$ Since $\delta\in \mathbb{F}_{q}$, it follows that $D(f)$ is a square in $\mathbb{F}_{q}$.
\end{proof}

     Now, we prove Theorem \ref{T3.1}.
    \begin{proof}[\textbf{Proof of \upshape{Theorem \ref{T3.1}}:}]
        We first note that no primitive element $\lambda\in\mathbb{F}_{q^n}$ can be a square in $\mathbb{F}_{q^n}$. To see this, suppose that $\lambda=\beta^2$ for some $\beta\in\mathbb{F}_{q^n}$. Then $\lambda^{(q^n-1)/2}=\beta^{q^n-1}=1$, which contradicts the fact that $\lambda$ generates the cyclic group $\mathbb{F}_{q^n}^{*}$.

        Now, let $q=3^k$, $k\in\mathbb{N}$, and consider the polynomial $f(x)=x^3+x^2+x+\lambda$. By Definition \ref{D3.1}, the discriminant of $f(x)$ is $D(f)=-\lambda$, since $\mathbb{F}_{q^n}$ has characteristic $3$. Furthermore, if $f(x)$ is a primitive polynomial over $\mathbb{F}_{q^n}$, then by Lemma \ref{L2.1}, $(-1)^3\lambda=-\lambda$ must be a primitive element of $\mathbb{F}_{q^n}$, and hence, it cannot be a square in $\mathbb{F}_{q^n}$. On the other hand, Proposition \ref{P3.1} states that the discriminant of every irreducible cubic polynomial over $\mathbb{F}_{q^n}$ is a square in $\mathbb{F}_{q^n}$. Therefore, $f(x)$ cannot be irreducible over $\mathbb{F}_{q^n}$. In particular, $f(x)$ cannot be primitive over $\mathbb{F}_{q^n}$.
    \end{proof}

    \begin{cor}\label{Cor3.1}
        Let $q=3$ and $n=2$. Then, the polynomial $x^3+x^2+x+\lambda$ is not primitive over $\mathbb{F}_{3^2}$, for any primitive $\lambda\in\mathbb{F}_{3^2}$. In particular, {\upshape Conjecture \ref{Con1.2}} of Ambrish Awasthi and Rajendra K. Sharma is not true for $q=3$ and $n=2$.
    \end{cor}
    \begin{proof}
        The result follows immediately from Theorem \ref{T3.1} by taking $q=3$ and $n=2$.
    \end{proof}

    \begin{rk}\label{R3.1}
        In {\upshape Theorem \ref{T2.1}}, we showed that if $q^n\equiv 3\pmod 4$ and $m$ is odd, then there does not exist a primitive polynomial $f(x)=g(x)+\lambda$ of degree $m$ over $\mathbb{F}_{q^n}$. On the other hand, {\upshape Theorem \ref{T3.1}} shows that there does not exist a primitive polynomial of the form $x^3+x^2+x+\lambda$ over $\mathbb{F}_{q^n}$ when $q=3^k$, $k\geq 1$. These results and the computational verifications suggest that a primitive polynomial of the form $x^3+x^2+x+\lambda$ may exist when $q$ is an odd prime power that is not a power of $3$ and $q^n\equiv 1 \pmod 4$. In fact, applying {\upshape Algorithm \ref{Algo2}} to such pairs $(q,n)$, we find that a primitive polynomial of the required form may exist. The results for several such pairs $(q,n)$ are presented in {\upshape Table \ref{table4}}.    
    \end{rk}

    \begin{algorithm}[h!]
    \caption{Testing Primitivity of $x^3+x^2+x+\lambda$}
    \label{Algo2}
    \begin{algorithmic}[1]
        \Require $q,n$
        \Ensure The number of primitive polynomials $x^3+x^2+x+\lambda$ over $\mathbb{F}_{q^n}$ for primitive $\lambda\in\mathbb{F}_{q^n}$

        \State Construct $\mathbb{F}_{q^n}$ and the polynomial ring $\mathbb{F}_{q^n}[x]$
        \State Set $\texttt{primitive}=0$

        \For{$\lambda$ in $\mathbb{F}_{q^n}$}
            \If{$\mathrm{ord}(\lambda)=q^n-1$}
                \State Set $f(x)=x^3+x^2+x+\lambda$
                \If{$f(x)$ is primitive over $\mathbb{F}_{q^n}$}
                    \State $\texttt{primitive}\gets\texttt{primitive}+1$
                \EndIf
            \EndIf
        \EndFor
        \State \textbf{return} $\texttt{primitive}$
    \end{algorithmic}
    \end{algorithm}

    \begin{table}[h!]
        \centering
        \caption{Primitive polynomials of the form $x^3+x^2+x+\lambda$ for selected pairs $(q,n)$ with odd prime power $q$ such that $q\neq 3^k$ and $q^n\equiv 1 \pmod 4$.}
        \label{table4}
        \begin{tabular}{c|>{\centering\arraybackslash}p{3cm}|>{\centering\arraybackslash}p{4.8cm}}
        \hline
        $(q,n)$& \# Primitive $\lambda\in\mathbb{F}_{q^{n}}$& \# Primitive Polynomials $x^3+x^2+x+\lambda$\\
        \hline
        $(5,1)$ & $2$ & 1\\
        $(5,2)$ & $8$ & 4\\
        $(5,3)$ & $60$ & 21\\
        $(5,4)$ & $192$ & 32\\
        $(5,5)$ & $1400$ & 450\\
        $(5,6)$ & $4320$ & 1398\\
        $(7,2)$ & $16$ & 2\\
        $(7,4)$ & $640$ & 172\\
        $(11,2)$ & $32$ & 14\\
        $(11,4)$ & $3840$ & 940\\
        $(13,1)$ & $4$ & 1\\
        $(13,2)$ & $48$ & 10\\
        $(13,4)$ & $6144$ & 1980\\
        $(17,1)$ & $8$ & 3\\
        $(17,2)$ & $96$ & 26\\
        $(17,4)$ & $21504$ & 5712\\
        $(19,2)$ & $96$ & 28\\
        $(19,4)$ & $34560$ & 9088\\
        $(23,2)$ & $160$ & 40\\
        $(23,4)$ & $66560$ & 16400\\
        $(25,1)$ & $8$ & 4\\
        $(25,2)$ & $192$ & 32\\
        $(25,3)$ & $4320$ & 1398\\
        $(29,1)$ & $12$ & 4\\
        $(29,2)$ & $192$ & 50\\
        $(31,2)$ & $256$ & 62\\
        $(37,2)$ & $432$ & 112\\
        $(41,2)$ & $384$ & 132\\
        $(43,2)$ & $480$ & 138\\
        \hline
        \end{tabular}
    \end{table}

\subsection{When the fields are of characteristic $2$}\label{S4}    

    In this subsection, we establish a criterion for the non-existence of such cubic primitive polynomials when $q$ is an even prime power. We prove the following result.

    \begin{theorem}\label{T3.2}
        Let $q=2^k,\, k\geq 1$. Then, for every $n\geq 1$ such that $kn$ is odd, the polynomial $f(x)=x^3+x^2+x+\lambda$ is not primitive over $\mathbb{F}_{q^n}$, where $\lambda$ is a primitive element of  $\mathbb{F}_{q^n}$.
    \end{theorem}
    \begin{proof}
        Suppose, to the contrary, that $f(x)=x^3+x^2+x+\lambda$ is primitive over $\mathbb{F}_{q^n}$. Since every primitive polynomial is irreducible, $f(x)$ is irreducible over $\mathbb{F}_{q^n}$. Let $\alpha$ be a root of $f(x)$. Then $$\alpha^3+\alpha^2+\alpha+\lambda=0.$$ Since the characteristic is $2$, we have $$(\alpha+1)^3=\alpha^3+\alpha^2+\alpha+1=1+\lambda\in\mathbb{F}_{q^n}.$$ 
        Note that if $q^n=2$, then $\lambda=1$, and the polynomial $x^3+x^2+x+\lambda=x^3+x^2+x+1=(x+1)^3$ is reducible over $\mathbb{F}_{2}$. Hence, cannot be primitive.
        
        Now, let $q^n>2$, then $1+\lambda\neq 0$ in $\mathbb{F}_{q^n}$. Consequently, $\alpha+1$ is also non-zero and   $(\alpha+1)^{3(q^n-1)}=1.$ Let $d=\mathrm{ord}(\alpha+1)$. Then $$d\mid 3(q^n-1).$$
        On the other hand, since $f(x)$ is irreducible of degree $3$ over $\mathbb{F}_{q^n}$, we have $\alpha\in\mathbb{F}_{q^{3n}}.$ Hence $\alpha+1\in\mathbb{F}_{q^{3n}}^*$, and therefore $$d\mid q^{3n}-1=(q^n-1)(q^{2n}+q^n+1).$$ It follows that $$d\mid \gcd\left(3(q^n-1),(q^n-1)(q^{2n}+q^n+1)\right),$$ and hence $$d\mid (q^n-1)\gcd\left(3,q^{2n}+q^n+1\right).$$
        Now suppose that $kn$ is odd. Since $q^n=2^{kn}$, we have $$2^{kn}\equiv -1\pmod 3.$$ Thus $$q^{2n}+q^n+1\equiv (-1)^2+(-1)+1\equiv 1\pmod 3,$$ so that $$\gcd\left(3,q^{2n}+q^n+1\right)=1.$$ Therefore, $d\mid q^n-1.$ It follows that $(\alpha+1)^{q^n-1}=1,$ and hence $\alpha+1\in\mathbb{F}_{q^n}.$ Since $1\in\mathbb{F}_{q^n}$, we obtain $\alpha\in\mathbb{F}_{q^n},$ which contradicts the irreducibility of $f(x)$ over $\mathbb{F}_{q^n}$.

        Therefore, when $kn$ is odd, there is no primitive polynomial of the form $x^3+x^2+x+\lambda$ over $\mathbb{F}_{q^n}$ with $\lambda$ primitive in $\mathbb{F}_{q^n}$.
    \end{proof}

    \begin{cor}\label{Cor3.2}
        Let $q=2$. Then, for every odd $n$, the polynomial $f(x)=x^3+x^2+x+\lambda$ is not primitive over $\mathbb{F}_{2^n}$, where $\lambda$ is a primitive element of  $\mathbb{F}_{2^n}$. In particular, {\upshape Conjecture \ref{Con1.2}} of Ambrish Awasthi and Rajendra K. Sharma is not true for $q=2$ and any odd $n$.
    \end{cor}
    \begin{proof}
        The result follows immediately from Theorem \ref{T3.2} by taking $q=2$ and odd $n$.
    \end{proof}

    \begin{rk}\label{R3.2}
        In {\upshape Theorem \ref{T3.2}}, we note that the condition on $kn$ being odd is sufficient for the non-existence of such cubic polynomials. However, if $kn$ is even, the argument does not determine whether such polynomials exist. Using {\upshape Algorithm \ref{Algo2}}, the computations suggest that such a cubic primitive polynomial may exist. We verified it for several pairs $(q,n)$ and tabulated the results in {\upshape Table \ref{table5}}.
    \end{rk}
    
    \begin{table}[h!]
        \centering
        \caption{The number of primitive polynomials $x^3+x^2+x+\lambda$ over $\mathbb{F}_{2^{kn}}$ with $kn$ even}
        \label{table5}
        \begin{tabular}{c|>{\centering\arraybackslash}p{3cm}|>{\centering\arraybackslash}p{4.8cm}}
        \hline
        $(q,n)$& \# Primitive $\lambda\in\mathbb{F}_{2^{kn}}$& \# Primitive Polynomials $x^3+x^2+x+\lambda$\\
        \hline
        $(2,2)$ & 2 & 2\\
        $(2,4)$ & 8 & 4\\
        $(2,6)$ & 36 & 18\\
        $(2,8)$ & 128 & 64\\
        $(4,1)$ & 2 & 2\\
        $(4,2)$ & 8 & 4\\
        $(4,3)$ & 36 & 18\\
        $(4,4)$ & 128 & 64\\
        $(4,5)$ & 600 & 340\\
        $(8,2)$ & 36 & 18\\
        $(8,4)$ & 1728 & 1092\\
        $(16,1)$ & 8 & 4\\
        $(16,2)$ & 128 & 64\\
        $(16,3)$ & 1728 & 1092\\
        $(32,2)$ & 600 & 340\\
        \hline
        \end{tabular}
    \end{table}

    Motivated by Theorems \ref{T3.1} and \ref{T3.2}, and based upon the observation in Remarks \ref{R3.1} and \ref{R3.2}, we propose the following conjectures in the cubic polynomial case.

    \begin{con}\label{Con3.1}
        Let $q$ be an odd prime power that is not a power of $3$, and let $n\geq 1$. If $q^n\equiv 1 \pmod{4}$, then there exists a primitive polynomial $x^3+x^2+x+\lambda$ over $\mathbb{F}_{q^n}$, where $\lambda$ is a primitive element of $\mathbb{F}_{q^n}$.
    \end{con}

    \begin{con}\label{Con3.2}
        Let $q=2^k,\, k\geq 1$ and let $n\geq 1$ such that $kn$ is even. Then there exists a primitive polynomial $x^3+x^2+x+\lambda$ over $\mathbb{F}_{q^n}$, where $\lambda$ is a primitive element of $\mathbb{F}_{q^n}$.
    \end{con} 
    
\section{Primitive polynomials of the form $x^p+x+\lambda$}\label{S4}
        In this section, we discuss the existence of primitive polynomials of the form $x^p+x+\lambda$ over $\mathbb{F}_{p^n}$, where $p$ is a prime and $\lambda$ is a primitive element in $\mathbb{F}_{p^n}$. We divide our discussion into two subsection. First, we discuss the case $p=2$, and then the case odd prime $p>2$. In the latter case, we propose two conjectures related to this existence.  

    \subsection{The case $p=2$}
        When $p=2$, this problem reduces to the existence of primitive polynomials of the form $x^2+x+\lambda$ over $\mathbb{F}_{2^n}$, where $\lambda$ is primitive in $\mathbb{F}_{2^n}$, for $n\geq 1$. 
        
        For $n=1$, we have $\mathbb{F}_{2^n}=\mathbb{F}_2$ and $\mathbb{F}_{2^{2n}}=\mathbb{F}_4$. Since $\mathbb{F}_4^*$ is cyclic of order $3$, any $\alpha\in\mathbb{F}_4\setminus\mathbb{F}_2$ is primitive. In particular, $\alpha^3=1$ and $\alpha\neq1$, so $\alpha^2+\alpha+1=0.$ Hence, $\mathrm{Tr}_{\mathbb{F}_4/\mathbb{F}_2}(\alpha)=\alpha+\alpha^2=1.$ Moreover, $\alpha$ is a root of $x^2+x+1$, and since $\alpha$ is primitive in $\mathbb{F}_4$, its minimal polynomial $x^2+x+1$ is a primitive polynomial over $\mathbb{F}_2$.
        
        Thus, it remains to consider the case $n\geq 2$. In {\upshape \cite[Theorem 1]{cohen1990primitive}}, Cohen proved that, for $m\geq 2$, and for any $t\in\mathbb{F}_q$, where $q$ is a prime power, there exists a primitive element $\alpha\in\mathbb{F}_{q^m}$ satisfying $\mathrm{Tr}_{\mathbb{F}_{q^m}/\mathbb{F}_q}(\alpha)=t,$ provided that either $m\neq 2$ and $m\neq 3$, or, in the exceptional cases, $t\neq 0$. Equivalently, there exists a primitive polynomial of degree $m$ over $\mathbb{F}_q$ with trace $t$ under the same conditions. In our setting, we take $m=p=2,\, q=2^n,\, t=1\neq 0.$ Therefore, Cohen's theorem guarantees the existence of a primitive polynomial of degree $2$ over $\mathbb{F}_{2^n}$ with trace $1$. Such a polynomial necessarily has the form $x^2+x+\lambda$ for some $\lambda\in\mathbb{F}_{2^n}$. Moreover, since the polynomial is primitive, its constant term, which is its norm, must be primitive in $\mathbb{F}_{2^n}$. Hence $\lambda$ is primitive in $\mathbb{F}_{2^n}$. Consequently, for every $n\geq 1$, there exists a primitive polynomial of the form $x^2+x+\lambda$ over $\mathbb{F}_{2^n}$, where $\lambda$ is primitive in $\mathbb{F}_{2^n}$.       

    \subsection{The case $p>2$}

        In the case $p=2$, we establish that there always exists a primitive polynomial of the form $x^2+x+\lambda$ over $\mathbb{F}_{2^n}$ with $\lambda$ primitive in $\mathbb{F}_{2^n}$, for every $n\geq 1$. Motivated by this observation, we study the polynomials $x^p+x+\lambda$ over $\mathbb{F}_{p^n}$, and found some interesting results. In the following theorem, we show that such a polynomial is irreducible over $\mathbb{F}_{p^n}$ if and only if $n$ is even, provided $\lambda$ satisfies the following condition.
        \begin{equation}\label{delta}
            \delta:=\sum_{i=0}^{n-1}(-1)^i\lambda^{p^i}\neq 0,
        \end{equation}
        However, whether these polynomials are primitive remains to be established.

        \begin{prop}\label{P4.1}
            Let $p$ be an odd prime. Then, for every $\lambda\in\mathbb{F}_{p^n}$ satisfying
            $\delta\neq 0,$ the polynomial $x^p+x+\lambda$ is irreducible over $\mathbb{F}_{p^n}$ if and only if $n$ is an even positive integer.
        \end{prop}
        \begin{proof}
            
             Suppose $n$ is even. Let $\lambda\in\mathbb{F}_{p^n}$ be such that $\delta\neq 0$. We note that if $\lambda\in\mathbb{F}_p$, then $\lambda^{p^i}=\lambda$ for each $i$, and hence $\delta=\sum_{i=0}^{n-1}(-1)^i\lambda^{p^i}=0$. Therefore, $\lambda\not\in\mathbb{F}_p$. Now, suppose that $\alpha$ is a root of the polynomial $x^p+x+\lambda$. Then
            $$\alpha^p+\alpha+\lambda=0,$$ and hence
            $$\alpha^p=-\alpha-\lambda.$$ 
            If we apply the Frobenius map to $\alpha^p$, then we get 
            $$\alpha^{p^2}=(-\alpha-\lambda)^p=-\alpha^p-\lambda^p=\alpha+\lambda-\lambda^p.$$ 
            Similarly, $$\alpha^{p^3}=(\alpha+\lambda-\lambda^p)^p=\alpha^p+\lambda^p-\lambda^{p^2}=-\alpha-\lambda+\lambda^p-\lambda^{p^2}.$$ 
            Therefore, if $n$ is even then applying the Frobenius map successively to $\alpha^p=-\alpha-\lambda$, we obtain $$\alpha^{p^n}=\alpha+\delta.$$ 
           
            Notice that $\delta^{p^n}=\delta$. Since $\delta\neq 0$, the elements $\alpha$ and $\alpha^{p^n}$ are distinct. Now, we claim that the conjugates of $\alpha$ are given by the following equation. 
            $$\alpha^{{p^n}^j}=\alpha+j\delta \text{  \quad for } j=0,1,2,\ldots,p-1.$$
            Clearly, it is true for $j=0,1$. Now, by induction on $j$, 
            $$\alpha^{{p^n}^j}=(\alpha^{{p^n}^{j-1}})^{p^n}=(\alpha+(j-1)\delta)^{{p^n}}=\alpha^{p^n}+(j-1)\delta^{p^n}=\alpha+\delta+(j-1)\delta=\alpha+j\delta.$$
            Also, notice that if $0\leq r\neq s\leq p-1$, then $\alpha+r\delta\neq \alpha+s\delta$. Hence, all these conjugates are distinct. Since $\alpha$ has $p$ distinct conjugates over $\mathbb{F}_{p^n}$, the polynomial $x^p+x+\lambda$ is irreducible over $\mathbb{F}_{p^n}$.

            Conversely, let the polynomial $x^p+x+\lambda$ is irreducible over $\mathbb{F}_{p^n}$ for some $\lambda$ satisfying Condition \eqref{delta}, and $n$ is odd. Then applying the Frobenius map repeatedly, we obtain $\alpha^{{p^{n}}}=-\alpha-\delta$, and hence $$\alpha^{{p^{2n}}}=(-\alpha-\delta)^{p^n}=-\alpha^{p^n}-\delta^{p^n}=\alpha.$$ 
            This shows that there are at most two distinct conjugates of $\alpha$ over $\mathbb{F}_{p^n}$, and the minimal polynomial of $\alpha$ over $\mathbb{F}_{p^n}$ has degree at most $2$. Thus, the polynomial $x^p+x+\lambda$ is reducible unless $p=2$. Since we have assumed $p$ to be an odd prime, the polynomial $x^p+x+\lambda$ is never irreducible over $\mathbb{F}_{p^n}$.
        \end{proof}

        \begin{rk}\label{R4.1}
        {\upshape Proposition \ref{P4.1}} shows that, if $n$ is odd or $\delta=0$, then the polynomial $x^p+x+\lambda$ is not irreducible over $\mathbb{F}_{p^n}$, consequently, is not primitive over $\mathbb{F}_{p^n}$. This can also be verified using {\upshape Algorithm \ref{Algo3}}. For example, for $p=5$ and $n=3$, we find that there are $60$ primitive elements $\lambda\in\mathbb{F}_{5^3}$ satisfying {\upshape Condition \eqref{delta}}; however, for none of these $\lambda$ is the polynomial $x^p+x+\lambda$ primitive. For $p=5$ and $n=5$, there are $1400$ primitive $\lambda\in\mathbb{F}_{5^5}$ satisfying {\upshape Condition \eqref{delta}} but no polynomial $x^p+x+\lambda$ is found to be primitive over $\mathbb{F}_{5^5}$. Similarly, for $p=7$ and $n=5$, we find that there are $5600$ primitive elements $\lambda\in\mathbb{F}_{7^5}$ satisfying {\upshape Condition \eqref{delta}} but the polynomial $x^p+x+\lambda$ is not primitive for any such $\lambda$. 
     \end{rk}   

     \begin{algorithm}[h!]
    \caption{{Testing Primitivity of $x^p+x+\lambda$} \label{Algo3}}
    \begin{algorithmic}[1]
        \Require $p,n$
        \Ensure The number of primitive elements $\lambda$ satisfying $\delta\neq0$, and the number of corresponding primitive polynomials $x^p+x+\lambda$

        \State Construct the field $\mathbb{F}_{p^n}$ and the polynomial ring $\mathbb{F}_{p^n}[x]$
        \State Set $\texttt{total}=0$ and $\texttt{primitive}=0$

        \For{$\lambda$ in $\mathbb{F}_{p^n}$}
            \If{$\mathrm{ord}(\lambda)={p^n}-1$}
                \State Compute
                $$
                \delta=\sum_{i=0}^{n-1}(-1)^i \lambda^{p^i}
                $$
                \If{$\delta\neq 0$}
                    \State $\texttt{total}\gets\texttt{total}+1$
                    \State Set $f(x)=x^p+x+\lambda$
                    \If{$f(x)$ is primitive over $\mathbb{F}_{p^n}$}
                        \State $\texttt{primitive}\gets\texttt{primitive}+1$
                    \EndIf
                \EndIf
            \EndIf
        \EndFor
        \State \textbf{return} $\texttt{total},\texttt{primitive}$
    \end{algorithmic}
\end{algorithm}

        Based on our experimental results, we propose the following conjectures related to the existence of primitive polynomials of the form $x^p+x+\lambda$.
        
        \begin{con}\label{Con4.1}
            Let $p$ be an odd prime, and let $n$ be an even positive integer. Then there exists a primitive element $\lambda\in\mathbb{F}_{p^n}$ satisfying
            \begin{equation*}
                \delta=\sum_{i=0}^{n-1}(-1)^i\lambda^{p^i}\neq 0,
            \end{equation*}
            such that the polynomial $x^p+x+\lambda$ is primitive over $\mathbb{F}_{p^n}$.
        \end{con}

        If this conjecture holds, then {\upshape Conjecture \ref{Con1.1}} of Ambrish Awasthi and Rajendra K. Sharma holds for the particular choice $g(x)=x^p+x$ with $m=p$ and even $n$.

\begin{table}[h!]
        \centering
        \caption{Computational verification of Conjecture \ref{Con4.1} for selected values of $(p,n)$.}
        \label{table6}
        \begin{tabular}{
        c|>{\centering\arraybackslash}p{2.3cm}|>{\centering\arraybackslash}p{3.5cm}|>{\centering\arraybackslash}p{4.8cm}}
        \hline
        $(p,n)$ & \# Primitive $\lambda\in\mathbb{F}_{p^n}$ & \# Primitive $\lambda$ with $\delta\neq0$& \# Primitive Polynomials $x^p+x+\lambda$\\
        \hline
        $(3,2)$& 4 & 4 & 4\\
        $(5,2)$& 8 & 8 & 8\\
        $(7,2)$& 16 & 16 & 16\\
        $(11,2)$& 32 & 32 & 32\\
        $(13,2)$& 48 & 48 & 48\\
        $(17,2)$& 96 & 96 & 96\\
        $(19,2)$& 96 & 96 & 96\\
        $(23,2)$& 160 & 160 & 160\\
        $(29,2)$& 192 & 192 & 192\\
        $(31,2)$& 256 & 256 & 256\\
        $(37,2)$& 432 & 432 & 432\\
        $(3,4)$& 32 & 24 & 24\\
        $(5,4)$& 192 & 160 & 160\\
        $(7,4)$& 640 & 560 & 560\\
        $(11,4)$& 3840 & 3520 & 3264\\
        $(13,4)$& 6144 & 5696 & 5664\\
        $(3,8)$& 2560 & 1728 & 1424\\
        \hline
        \end{tabular}
    \end{table}

\section{Conclusion}

        In this paper, we investigated primitive polynomials $f(x)$ with a primitive constant term. We established non-existence results showing that, when $q^n\equiv3\pmod{4}$ and $m$ is odd, no such primitive polynomial exists, which disproves Conjecture \ref{Con1.1}. Further, for the cubic case, we proved that $x^3+x^2+x+\lambda$ is never primitive over $\mathbb{F}_{q^n}$ for primitive $\lambda$ when either $q=3^k$, or $q=2^k$ with $kn$ odd. We then formulated conjectures for the remaining cases and provided computational evidence in their support.

        We also studied the family $x^p+x+\lambda$ over $\mathbb{F}_{p^n}$. For $p=2$, we proved that there always exists such a primitive polynomial, and for odd primes $p$ and even $n$, we proved that the condition $\sum_{i=0}^{n-1}(-1)^i\lambda^{p^i}\neq0$ ensures irreducibility of $x^p+x+\lambda$ over $\mathbb{F}_{p^n}$. However, the primitivity of such polynomials is yet to be established. Based on computational evidence, we formulated a conjecture on the existence of a primitive element $\lambda$ for which this polynomial is primitive. If this conjecture is true, it would establish Conjecture \ref{Con1.1} of Ambrish Awasthi and Rajendra K. Sharma for $m=p$ and even $n$.
 
\section{Acknowledgments}
        We sincerely thank Dr. Sharwan K. Tiwari for his valuable suggestions and fruitful discussions.

\section{Declaration of generative AI and AI-assisted technologies in the manuscript preparation process} 

        During the preparation of this work, the author used ChatGPT solely for paraphrasing purposes. The author carefully reviewed and edited the generated text as necessary and assumes full responsibility for the accuracy, originality, and content of the published article.

\bibliographystyle{elsarticle-harv}
\bibliography{bib.bib}

\end{document}